\documentclass[11pt,reqno,a4paper]{amsart}

\usepackage{tabularx, hyperref}
\usepackage{amssymb} \usepackage{amsfonts} \usepackage{amsmath}
\usepackage{amsthm} \usepackage{epsfig, subfig}
\usepackage{ amscd, amsxtra, latexsym}
\usepackage[all]{xy}
\usepackage{caption}
\usepackage{enumerate}
\usepackage{color}
\usepackage{comment}

\newtheorem{lemma}{Lemma}[section]
\newtheorem{thm}[lemma]{Theorem}
\newtheorem{prop}[lemma]{Proposition}
\newtheorem{cor}[lemma]{Corollary}

\theoremstyle{definition}
\newtheorem{defn}[lemma]{Definition}

\newtheorem{rem}[lemma]{Remark}

\theoremstyle{definition}

\definecolor{darkgreen}{cmyk}{1,0,1,.2}

\newcommand{\R} {\ensuremath {\mathbb{R}}}

\newcommand{\Z} {\ensuremath {\mathbb{Z}}}

\newcommand{\matP} {\ensuremath {\mathbb{P}}}

\newcommand{\h}{\hookrightarrow_{h}}

\newcommand{\e}{\varepsilon }
\newcommand{\nei}[2]{\mathcal{N}_{#2}(#1)}

\address{Mathematic Department, Technion, Haifa, 32000, Israel}
\email{hartnick@tx.technion.ac.il}
\address{Department of Mathematics, ETH Zurich, 8092 Zurich, Switzerland}
\email{sisto@math.ethz.ch}

\begin{document}

\title[Bounded cohomology and virtually free hyperbolically embedded subgroups]{Bounded cohomology and virtually free hyperbolically embedded subgroups}
\author{Tobias Hartnick}
\author{Alessandro Sisto}

\begin{abstract} Using a probabilistic argument we show that the second bounded cohomology of an acylindrically hyperbolic group $G$ (e.g., a non-elementary hyperbolic or relatively hyperbolic group, non-exceptional mapping class group, ${\rm Out}(F_n)$, \dots) embeds via the natural restriction maps into the inverse limit of the second bounded cohomologies of its virtually free subgroups, and in fact even into the inverse limit of the second bounded cohomologies of its hyperbolically  embedded virtually free subgroups.
This result is new and non-trivial even in the case where $G$ is a (non-free) hyperbolic group. The corresponding statement fails in general for the third bounded cohomology, even for surface groups. 
%As a consequence of our main result, any surjection onto an acylindrically hyperbolic group induces an injection at the level of second bounded cohomologies.
\end{abstract}

\maketitle

 \section{Introduction}
Thirty-five years after Gromov's groundbreaking paper \cite{Gromov} which jump-started the subject, bounded cohomology of discrete groups is by-and-large as computationally elusive at is was back then. We still do not know the bounded cohomology ring of a single group, except for those where it is trivial. Much work has been done on the second bounded cohomology of groups with various weak negative curvature properties over the last 20 years, in particular establishing infinite-dimensionality of $H^2_b(G; \R)$ for free groups, surface groups, non-elementary hyperbolic and relatively hyperbolic groups, non-exceptional mapping class groups, ${\rm Out}(F_n)$ and, generalizing all previous examples, acylindrically hyperbolic groups \cite{Brooks,BrooksSeries,EF, F, BF,HO}, but we are very far from a complete understanding of $H^2_b(G; \R)$ even for hyperbolic groups.  

By far the best understood case is that of a finitely-generated non-abelian free group $F$. In this case we know  due to a result of Grigorchuk \cite{Grigorchuk} that the Brooks counting quasimorphisms generate a dense subspace of $H^2_b(F; \R)$ (albeit only in the non-complete topology of pointwise convergence of homogeneous representatives). The linear relations between arbitrary Brooks counting quasimorphisms in $H^2_b(F; \R)$ and various explicit bases for the dense subspace they generate are known \cite{HartnickTalambutsa}, and there are some very modest beginnings of a study of the structure of $H^2_b(F; \R)$ as an ${\rm Out}(F)$-module \cite{HartnickSchweitzer, Hase}. Nothing comparable is known for any other groups, not even surface groups.
\begin{comment}
Although it follows from the work of Buger and Monod that there exists for every $n \geq 1$ a finitely-presented group $G$ with $H^2_b(G; \R) \cong \R^n$, in almost all studied examples the second bounded cohomology is either trivial or infinite-dimensional. One of the main themes over the last few decades was to establish infinite-dimensionality of $H^2_b(G; \R)$ for large classes of group $G$ by constructing infinitely many quasimorphisms on $G$. These constructions lead to some explicit cocycles in various classes of groups, but we have no idea how much of the second bounded cohomology these amount to. In fact, the only groups for which we have some modest understanding of the structures of $H^2_b(G; \R)$ are free groups. 

Indeed, for a finitely-generated free group $F$ we know due to a result of Grigorchuk that the Brooks counting quasimorphisms generate a dense subspace of $H^2_b(F; \R)$ (albeit only in the non-complete topology of pointwise convergence of homogeneous representatives). The linear relations between arbitrary Brooks counting quasimorphisms in $H^2_b(F; \R)$ and various explicit bases for the dense subspace they generate are known, and there are some very modest beginnings of a study of the structure of $H^2_b(F; \R)$ as an ${\rm Out}(F)$-module. Nothing comparable is known for any other groups, not even surface groups.
\end{comment}

This note is a first attempt to understand the finer structure of $H^2_b(G; \R)$ for more general groups $G$ by relating $H^2_b(G; \R)$ to the bounded cohomology of certain (virtually) free subgroups of $G$. Given any collection $\mathcal S$ of subgroups of $G$, the restriction maps 
\[{\rm res}_H: H^2_b(G; \R) \to H^2_b(H; \R) \quad (H \in \mathcal S)\] combine into a map
\begin{equation}\label{iS}
{\rm res}_{\mathcal S}: H^2_b(G; \R) \to \underset{H \in \mathcal S}{\lim_{\leftarrow}} H^2_b(H; \R).
\end{equation}
The main idea behind our approach is that if one can show for some class $\mathcal S$ that this map is injective/surjective, then one obtains upper/lower bounds on $H^2_b(G; \R)$ in terms of the $H^2_b(H; \R)$ and maps between them. If $\mathcal S$ consists moreover of free groups, then one can use our structural knowledge about $H^2_b(F; \R)$ to study $H^2_b(G; \R)$. The same strategy can be applied to study the exact second bounded cohomology $EH^2_b(G; \R)$, i.e. the kernel of the comparison map $H^2_b(G; \R) \to H^2(G; \R)$.

We will study the map ${\rm res}_{\mathcal S}$ for groups $G$ which admit non-trivial hyperbolically embedded subgroups (and thus, a posteriori, hyperbolically embedded virtually \emph{free} subgroups) in the sense of Dahmani, Guirardel and Osin \cite{DGO}. The class of such groups is also known as acylindrically hyperbolic groups, and it is a very wide class \cite{Osin}. We will revisit the relevant definitions in Subsection \ref{SecAHG}. For the purposes of this introduction it is enough if the reader has in mind the case where $G$ is a Gromov-hyperbolic group. Then a subgroup $H < G$ is \emph{hyperbolically embedded} if it is malnormal and quasi-convex (because in this case $G$ is hyperbolic relative to $H$ \cite{Bowditch}). In general acylindrically hyperbolic groups, the notion will depend on a (typically infinite) generating set of $G$. Its relevance concerning bounded cohomology is apparent from the following extension result of Hull and Osin \cite{HO}.
\begin{thm}[Hull--Osin]
If $G$ is acylindrically hyperbolic and $\mathcal S=\{H\}$ for some hyperbolically embedded subgroup $H$, then ${\rm res}_{\mathcal S}: EH^2_b(G; \R) \to EH^2_b(H; \R)$ is surjective.
\end{thm}
By choosing $H$ to be a virtually non-abelian free group this shows that $\dim EH^2_b(G; \R) = \infty$. However, in this case the kernel of $\iota_{\mathcal S}$ will always be infinite-dimensional unless $G$ itself is virtually free. More generally, we have the following theorem \cite{FPS}:
\begin{thm}[Frigerio--Pozzetti--Sisto]
Let $\mathcal{S}$ be a finite hyperbolically embedded family of infinite index virtually free subgroups of an acylindrically hyperbolic group $G$. Then the map ${\rm res}_{\mathcal S}$ from \eqref{iS} has infinite-dimensional kernel.  
\end{thm}
It is then natural to ask whether one can recover $H^2_b(G; \R)$ using an infinite family of hyperbolically embedded virtually free subgroups. The following theorem, which is the main result of this note, answers this question in the positive. Given an acylindrically hyperbolic group $G$ we shall denote by ${\rm HE}_2(G)$ the collection of all virtually free-on-two-generators hyperbolically embedded subgroups of $G$.
%answer to this question, let us call a generating set $Y$ of an acylindrically hyperbolic group $G$ \emph{acylindrical} if there exists some virtually free group $F$ with $F \hookrightarrow_h (G, Y)$. Given such a generating set, we denote by $\mathcal{HE}_2(G, Y)$ the collection of all virtually free-on-two-generators subgroups of $G$ that are hyperbolically embedded in $(G, Y)$. 
\begin{thm}\label{ThmMain} Let $G$ be an acylindrically hyperbolic group and let 
$\mathcal S$ be a collection of subgroups of $G$ containing ${\rm HE}_2(G)$. Then the natural restriction maps
\[
{\rm res}_{\mathcal S}: H^2_b(G; \R) \to \underset{H \in \mathcal S}{\lim_{\leftarrow}} H^2_b(H; \R) \quad \text{and} \quad {\rm res}_{\mathcal S}: EH^2_b(G; \R) \to \underset{H \in \mathcal S}{\lim_{\leftarrow}} EH^2_b(H; \R) 
\]
are both injective.
\end{thm}
In fact, we will prove a stronger statement in Theorem \ref{MainRefinement}, where we fix the generating set with respect to which the subgroups we consider are hyperbolically embedded. We emphasize that the theorem is non-trivial even if one chooses $\mathcal S$ as the collection of all virtually free subgroups.

It would be of great interest to determine the images of these maps. A necessary condition for $(\alpha_H)_{H \in \mathcal S}$ to lie in the image of $\iota_{\mathcal S}$ is that the Gromov norm of the cocycles $\alpha_H$ are bounded uniformly over $H \in \mathcal S$, but this may hold automatically for a sufficiently large collection $\mathcal S$. We are not aware of any other obvious obstructions for a family of cohomology classes in the inverse limit to extend.

To prove Theorem \ref{ThmMain} we have to show that a given cohomology class $\alpha \in H^2_b(G; \R)$ restricts non-trivially to some $F \in {\rm HE}_2(G)$. Our proof is based on two ideas. The first idea is to choose $F$ at random. Here we rely heavily on a theorem of Maher and Sisto \cite{MS} which guarantees that two random elements of $G$ virtually generate an element of ${\rm HE}_2(G)$ if chosen according to a lazy simple random walk on $G$. We then need to show that $\alpha$ restricts non-trivially to a random subgroup with positive probability. In general, it is hard to decide whether a given cohomology class restricts non-trivially to a subgroup. Indeed, we have to check whether the restriction of a  cocycle representing $\alpha$ is a coboundary, and this is not a pointwise, but a global condition. Checking such global conditions on random subgroups seems impossible. Our second idea is thus to circumvent this problem in degree $2$ by using the notion of a restriction-homogeneous $2$-cocycle, due to Bouarich \cite{Bouarich95} (who calls them ``homogeneous''). Bouarich proved that every bounded cohomology class $\alpha$ in degree $2$ is represented by a unique restriction-homogeneous $2$-cocycle $c$, and we will show that $\alpha$ restricts trivially to a free subgroup if and only if the restriction of $c$ vanishes. Non-vanishing of the cocycle $c$ is a pointwise condition, whose probability on a random subgroup can now be checked. 

We point out that in order to prove the version of Theorem \ref{ThmMain} for $EH^2_b(G)$, we only need the more classical notion of a homogeneous quasimorphism rather than that of a restriction-homogeneous cocycle.

 Bouarich's result has no counterpart in higher degrees, and thus our argument does not extend beyond degree $2$. In fact, it cannot extend in general for the following reason:
 \begin{prop}\label{PropDeg3} Let $\Gamma_g = \pi_1(\Sigma_g)$ be the fundamental group of a closed oriented surface of genus $g \geq 2$. Then there exists a class $\alpha \in H^3_b(\Gamma_g; \R)$ which restricts trivially to every free subgroup of $\Gamma_g$. In fact, such a class can be constructed as the restriction to the fiber of the volume class of a hyperbolic mapping torus.
 \end{prop}

This article is organized as follows: In Section \ref{SecBouarich} we explain our notation concerning (bounded) group cohomology and then recall Bouarich's notion of a restriction-homogeneous $2$-cocycle. We then show how non-triviality of restrictions of classes in $H^2_b$ to free subgroups translates into a pointwise condition when using restriction-homogeneous representatives. In Section \ref{sec:twisted} we show that the probability of a restriction-homogeneous $2$-cocycle to be positive on a triple of points is non-zero with positive probability along a lazy random walk. In Section \ref{SecLRMAHG} we recall the necessary background on hyperbolically embedded subgroups and acylindrically hyperbolic groups. We then use the aforementioned result of Maher and Sisto on lazy random walks in such groups to deduce our main result. Finally, in Section \ref{SecDegree3} we establish Proposition \ref{PropDeg3}.

\subsection*{Acknowledgments} This work was initiated during the conference ``Young geometric group theory IV'' in Spa (Belgium), 2015.

This material is based upon work supported by the National Science
Foundation under Grant No. DMS-1440140 while the authors were in
residence at the Mathematical Sciences Research Institute in Berkeley,
California, during the Fall 2016 semester.

T. H. was moreover partially supported by ISF grant No. 535/14.

\section{Restriction-homogeneous $2$-cocycles}\label{SecBouarich}

In this section we collect background material on bounded cohomology, and then we discuss restriction-homogeneous cocycles. The reader interested only in the $EH^2_b(G)$ case of Theorem \ref{ThmMain} can ignore the definition of a restriction-homogeneous cocycle and in later sections replace the term ``restriction-homogeneous cocycle'' by ``differential of a homogeneous quasimorphism''.

\subsection{Notations concerning (bounded) group cohomology}
Throughout this article we will use the following conventions and notations concerning (bounded) group cohomology (see e.g. \cite{Frigerio}). $G$ will always denote a countable group.
The group cohomology $H^\bullet(G; \R)$ of $G$ with (trivial) $\R$-coefficients can be computed by either the homogeneous or the inhomogeneous bar resolution. Cochains in the former resolution are given by $C^n(G) := {\rm Map}(G^{n+1}; \R)^G$ with differential given by 
\[
d^nc(g_0, \dots, g_{n+1}) := \sum_{j=0}^{n+1} (-1)^j c(g_0, \dots, \widehat{g_j}, \dots g_{n+1}).
\]
We use the notations $Z^n(G) := \ker(d^n)$ and $B^n(G) := {\rm im}(d^{n-1})$ %\quad\text{and}\quad H^n(G;\R) := Z^n(G)/B^n(G),
for $n$-cocycles, respectively $n$-coboundaries in this resolution. The inhomogeneous bar resolution has cochains given by $\overline{C}^n(G) := {\rm Map}(G^n, \R)$ with differential given by 
\[
\delta f(g_1, \dots, g_{n+1}) := f(g_2, \dots, g_{n+1})-\sum_{j=1}^n (-1)^{j+1} f(g_1, \dots, g_jg_{j+1}, \dots, g_n) + (-1)^{n} f(g_1, \dots, g_n),
\]
and we denote $n$-cocycles and $n$-coboundaries in this resolution by $\overline{Z}^n(G) := \ker(\delta^n)$ and $\overline{B}^n(G) := {\rm im}(\delta^{n-1})$. The map
\begin{equation}\label{ChainIso}
\iota_\bullet: (C^\bullet(G), d^\bullet) \to (\overline{C}^\bullet(G), \delta^\bullet), \quad (\iota_n c)(g_1, \dots, g_n) := c(e, g_1, g_1g_2, \dots, g_1g_2 \cdots g_n)
\end{equation}
defines an isomorphism of cochain complexes with inverse given by $(\iota_n^{-1}f)(g_0, \dots, g_n):=f(g_0^{-1}g_1, \dots, g_{n-1}^{-1}g_n)$. In particular,
\[
H^n(G; \R) = H^n (C^\bullet(G), d^\bullet) = H^n (\overline{C}^\bullet(G), \delta^\bullet).
\]
We denote by $C^\bullet_b(G) \subset C^\bullet(G)$ and $\overline{C}^\bullet_b(G) \subset \overline{C}^\bullet(G)$ the respective subcomplexes of bounded cochains and define  ${Z}_b^n(G)$, ${B}_b^n(G)$, $\overline{Z}_b^n(G)$ and $\overline{B}_b^n(G)$ accordingly. Then $\iota_n$ restricts to an isomorphism $C^n_b(G) \to \overline{C}^n_b(G)$ and the bounded group cohomology of $G$ with (trivial) $\R$-coefficients is defined as
\[
H_b^n(G; \R) := H^n (C_b^\bullet(G), d^\bullet) = H^n (\overline{C}_b^\bullet(G), \delta^\bullet).
\]

\subsection{$2$-cocycles and quasimorphisms} 
Let us specialize the definitions of the previous subsection to the case $n=2$.
By definition, a function $c: G^3 \to \R$ is contained in $Z^2(G; \R)$ iff it is $G$-invariant and satisfies the \emph{cocycle identity}
\begin{equation}\label{Coboundary}
d^2c(g_0, g_1, g_2, g_3) = c(g_1, g_2, g_3) - c(g_0, g_2, g_3) + c(g_0, g_1, g_3) - c(g_0, g_1, g_2) = 0.
\end{equation}
It is contained in $Z^2_b(G; \R)$ if it is moreover bounded. Geometrically, the cocycle condition says that the sum of the values (with signs) of $c$ over the faces of any tetrahedron in the Milnor model of the classifying space of $G$ is $0$. Most of the lemmas in Section \ref{sec:twisted} were inspired by this geometric interpretation.

A function $h: G^2 \to \R$ is contained in $\overline{Z}^2(G; \R)$ provided
\[
\delta^2h(g_1, g_2, g_3) = h(g_2, g_3) - h(g_1g_2, g_3) + h(g_1, g_2g_3) - h(g_1, g_2) = 0.
\]
A function $f \in \overline{C}^1(G) = {\rm Map}(G; \R)$ is called a quasimorphism if $\delta^1f$ is bounded, and we denote the space of such functions by $Q(G)$. Note that
$\delta^1f(g_1, g_2) = f(g_2)-f(g_1g_2) + f(g_1)$ and thus $f \in Q(G)$ if and only if
\[
D(f) := \sup_{g_1, g_2 \in G} \, | f(g_1g_2) - f(g_1)-f(g_2)| < \infty.
\]

We then have an isomorphism \cite[Prop. 2.8]{Frigerio}
\[
i: \frac{Q(G)}{{\rm Hom}(G; \R) \oplus \ell^\infty(G)} \to EH^2_b(G; \R), \quad f \mapsto [\delta^1 f].
\]
A cocycle $c \in Z^2(G)$ representing $i(f)$ is given by $c_f := \iota_2^{-1}(\delta^1f)$ where $\iota_2$ is defined by \eqref{ChainIso}. Explicitly, 
\begin{equation}\label{cf}
c_f(g_0, g_1, g_2) = f(g_0^{-1}g_1) + f(g_1^{-1}g_2) + f(g_2^{-1}g_0).
\end{equation}

\subsection{Restriction-homogeneous $2$-cocycles} Recall that a quasimorphism $f \in Q(G)$ is called \emph{homogeneous} provided $f(g^n)=nf(g)$ for each $g\in G$ and $n\in\mathbb N$. We denote by $\mathcal H(G) \subset Q(G)$ the subspace of homogeneous quasimorphisms. Every $f \in \mathcal H(G)$ is automatically conjugation-invariant and restricts to a homomorphism on abelian (in fact, amenable) subgroups. In particular, $f(g^n) = nf(g)$ for all $n \in \mathbb Z$. Equivalently, 
\[
\delta^1 f(g^n, g^m) = f(g^m) - f(g^{n+m}) + f(g^n) = 0
\]
for all $m,n \in \Z$. In analogy with this case, Bouarich defined \cite{Bouarich95}:
\begin{defn}[Bouarich] An inhomogeneous $2$-cocycle $h \in \overline{Z}^2(G; \R)$ is called \emph{restriction-homogeneous} if it satisfies $h(g^n, g^m) = 0$ for all $g\in G$ and $m,n \in \Z$.
%$c \in Z^2(G; \R)$ saisfying the equivalent conditions of Proposition \ref{PropHomogeneous} is 
\end{defn}
Thus by definition a quasimorphism $f$ is homogeneous if and only if the bounded $2$-cocycle $h:=\delta^1f$ is restriction-homogeneous. 
\begin{rem} Bouarich calls ``homogeneous $2$-cocycle'' what we call a ``restriction-homogeneous $2$-cocycle'' here. We changed his original terminology to avoid
linguistically awkward constructions like ``homogeneous inhomogeneous $2$-cocycle''.
\end{rem}

\begin{lemma}\label{LemmaHomogeneous}
Given a group $G$ and $h \in \overline{Z}^2(G; \R)$, the following are equivalent.
\begin{enumerate}[(i)]
\item $h$ is restriction-homogeneous.
\item $h|_Z \equiv 0$ for every cyclic subgroup $Z < G$.
\item If $H < G$ and $[h|_H] \in EH^2_b(H; \R)$, then $h|_H = \delta^1 f$ for some $f \in \mathcal H(H)$.
\end{enumerate}
\end{lemma}
\begin{proof} (ii) is just a reformulation of (i), so let us show that (ii)$\Leftrightarrow$(iii). First assume (ii) and let $H<G$ as in (iii). Then $h|_H = \delta^1f$ for some $f \in Q(H)$. If $Z < H$ is cyclic, then $\delta^1f|_Z = h|_Z = 0$, hence $f \in \mathcal H(H)$ as claimed. Conversely, if (iii) holds and $Z < G$ is cyclic, then $h|_Z \in EH^2_b(Z; \R)$, hence by assumption $h|_Z = \delta^1 f$ for some $f \in \mathcal H(Z)$. But for a cyclic group we have $\mathcal H(Z) = {\rm Hom}(Z; \R)$ and thus $h|_Z = \delta^1 f = 0$. This finishes the proof.
\end{proof}
Since we prefer to work with cocycles in $Z^2(G; \R)$, let us call a cocycle $c \in Z^2(G; \R)$ \emph{restriction-homogeneous} if the corresponding cocycle $h = \iota_n(c) \in \overline{Z}^2(G; \R)$ is restriction-homogeneous. Then Lemma \ref{LemmaHomogeneous} translates as follows.
\begin{cor}\label{CorHomogeneous} Given a group $G$, a cocycle $c \in {Z}^2(G; \R)$ is restriction-homogeneous if and only if it satisfies one of the following properties.
\begin{enumerate}[(i)]
\item For all $m, n \in \Z$ and $g \in G$ we have $c(e,g^n, g^m) = 0$.
\item If $Z < G$ is cyclic, then $c|_Z = 0$.
\item If $H < G$ and $[c|_H] \in EH^2_b(H; \R)$, then $c|_H = c_f$ for some $f \in \mathcal H(H)$.\qed
\end{enumerate}
\end{cor}
The proof of the following lemma was sketched in \cite{Bouarich95}; for a detailed proof see \cite[Prop. 2.16]{Frigerio}.
\begin{lemma}[Bouarich]\label{LemmaBouarich} Let $G$ be a group. Then every class $\alpha \in H^2_b(G)$ admits a unique restriction-homogeneous representative.\qed
\end{lemma}
\begin{cor}\label{CorAlternating} Every restriction-homogeneous $2$-cocycle $c\in Z^2(G; \R)$ is alternating.
\end{cor}
\begin{proof} It is immediate from Property (ii) of Corollary \ref{CorHomogeneous} that if $c\in Z^2(G; \R)$ is restriction-homogeneous, then also
\[\widetilde{c}(g_0, g_1, g_2) := \sum_{\sigma \in {S_3}} (-1)^\sigma c(g_{\sigma_0}, g_{\sigma_1}, g_{\sigma_2})\]
is restriction-homogeneous. Moreover, $\widetilde{c}$ is alternating by construction and cohomologous to $c$. By uniqueness we have $c = \widetilde{c}$, i.e. $c$ is alternating.
\end{proof}

\subsection{Properties of restriction-homogeneous cocycles}
In general, checking non-triviality of restrictions of cohomology classes is difficult due to the presence of coboundaries. In degree $2$ checking non-triviality of such a restriction can in many cases be reduced to a pointwise condition, namely non-triviality of the restriction of a restriction-homogeneous representative of the given class. This observation, which has no counterpart in higher degree bounded cohomology, lies at the heart of our approach and can be formulated as follows.
\begin{lemma}\label{HomogeneousNonvanishing}
Let $G$ be a group, $H<G$ a subgroup and $c \in Z_b^2(G; \R)$ a bounded restriction-homogeneous $2$-cocycle. If $[c|_H]\in EH^2_b(H)$ and if there exist $h_0,h_1,h_2\in H$ with $c(h_0,h_1,h_2)\neq 0$, then $[c|_H] \in EH^2_b(H)$ is non-zero.
\end{lemma}
\begin{proof}
 Since $c$ is restriction-homogeneous we have $c|_H=c_f$ for some homogeneous quasimorphism $f$. If $[c|_H] = [c_f]$ was trivial, then $f$ was a homomorphism, hence $c_f \equiv 0$, contradicting our assumption.
\end{proof}
Note that if $H < G$ is a virtually free subgroup, then $H^2_b(H; \R) = EH^2_b(H; \R)$. We can thus record the following special case for later reference:
\begin{cor}\label{HomogeneousNonvanishing2} Let $G$ be a group, $H<G$ a virtually free subgroup and $c \in Z_b^2(G; \R)$ a bounded restriction-homogeneous $2$-cocycle defining a class $[c] \in H^2_b(G; \R)$. Then 
\[[c|_H] = 0 \quad \Leftrightarrow \quad c|_H = 0. \quad \qed\]
\end{cor}
For later use, we also record the following property of restriction-homogeneous cocycles.
\begin{lemma}\label{HomogeneousSum}
Let $G$ be a group and $c \in Z_b^2(G; \R)$ a bounded restriction-homogeneous $2$-cocycle. Then for all $g, h \in G$ and $N \in \mathbb N$ we have
\[c\left(g^{-N-1}, h, e\right) = \sum_{i=0}^N c\left(g^{-1}, g^ih, e\right).\]
\end{lemma}

\begin{proof} We argue by induction on $N$, the case $N=0$ being trivial. Assume that the statement holds for some $N$, and apply the cocycle identity \eqref{Coboundary} to the quadruple $(g^{-N-1},g^{-N-2}, h, e)$ to obtain
 \begin{equation}\label{EqHomSum}
  c(g^{-N-2}, h, e) = c(g^{-N-1}, h, e) -c(g^{-N-1},g^{-N-2},e)+c(g^{-N-1},g^{-N-2}, h).
 \end{equation}
 Since $c$ is restriction-homogeneous we have $c(g^{-N-1},g^{-N-2},e) = 0$. On the other hand, since $c$ is $G$-invariant and moreover alternating by Corollary \ref{CorAlternating}, we obtain
 \[
 c(g^{-N-1},g^{-N-2}, h) = c(e, g^{-1}, g^{N+1}h) = c(g^{-1},g^{N+1} h, e).
 \]
 Plugging both into \eqref{EqHomSum} and using the induction hypothesis we obtain
 \begin{align*}
  c(g^{-N-2}, h, e) &=  c(g^{-N-1}, h, e)+ c(g^{-1},g^{N+1} h, e) =  \sum_{i=0}^N c\left(g^{-1}, g^ih, e\right) +  c(g^{-1},g^{N+1} h, e)\\ 
  &= \sum_{i=0}^{N+1} c\left(g^{-1}, g^ih, e\right).
 \end{align*}
 This finishes the proof. 
\end{proof}

\section{Twisted triangles happen with positive probability}\label{sec:twisted}

\subsection{Setting and formulation of the main result}
Let $G$ be a finitely generated group and $S$ a finite symmetric generating set for $G$ containing the identity. We consider the measure $\mu = \mu_S := \frac{1}{|S|}\sum_{s\in S} \delta_s$. Let $(g_n)$ be a sequence of independent random variables on $G$ with distribution $\mu$ and let $x_n:=g_1 \cdots g_n$ be the product of $n$ randomly and independently chosen elements of $S$. We refer to the sequence $(x_n)$ of random variables as the \emph{lazy simple random walk} on $G$ associated to $S$. Here, laziness refers to the fact that at each step the random walk stays still with positive probability since $S$ contains the identity. The goal of this section is to prove the following theorem:
\begin{thm}
\label{posprobtwist}
Let $G$ be a finitely generated group, $c \in Z_b^2(G)$ restriction-homogeneous and assume that $[c] \neq 0 \in H^2_b(G; \R)$. Let $(x_n),(y_m)$ be independent lazy simple random walks on $G$. Then there exist $\epsilon, m_0, n_0>0$ so that for each $n\geq n_0$, $m\geq m_0$ we have
\[\matP(|c(x_n,y_m,e)|\geq \epsilon)\geq \epsilon.\]
\end{thm}
For the rest of this section we fix a group $G$ and a restriction-homogeneous $c \in Z_b^2(G)$ with $[c] \neq 0$. We then call a triple $(g,h,k) \in G^3$ \emph{twisted} provided $|c(g,h,k)| > 0$ and more precisely \emph{$\epsilon$-twisted} for some $\epsilon>0$ provided $|c(g,h,k)|\geq \epsilon$. Using this language, the theorem states that twisted triangles happen with positive probability. 

The reader only interested in the $EH^2_b$-version of the main theorem may moreover assume that $c = c_f$, where $f$ is a homogeneous quasimorphism on $G$ and $c_f$ is defined by \eqref{cf}.

\subsection{Propagation of twisted triangles}
Our starting point is the following immediate consequence of the cocycle identity \eqref{Coboundary}, which says that if one face of a tetrahedron is twisted, then at least one other face of that tetrahedron is also twisted (possibly with smaller constant).
\begin{lemma}\label{tripletwist}
 Assume $(g_0, g_1, g_2)$ is $3 \epsilon$-twisted. Then for every $h \in G$ at least one of $(g_1, g_2, h)$, $(g_0, g_2, h)$, $(g_0, g_1, h)$ is $\epsilon$-twisted.\qed
\end{lemma}
Since $[c]\neq 0$ and hence $c\neq 0$ we know that there exists at least \emph{some} twisted triangle; since $c$ is $G$-invariant we can find a twisted triangle whose last coordinate is the identity. Let us fix such a twisted triangle $(g_0, h_0, e)$ and set $\epsilon := \frac 1 3 |c(g_0,h_0,e)|$.
\begin{comment}
Since $[c]\neq 0$ and hence $c\neq 0$ we know that there exists at least \emph{some} $\epsilon$-twisted triangle for \emph{some} $\epsilon>0$. The point of the proof will be to construct many twisted triangle starting from a given one. Since $c \neq 0$ there exists at least one twisted triangle, say $(g_0, h_0, e)$. Denote 
\end{comment}
\begin{cor}\label{RWAlternative}
\hspace{2mm}
\begin{enumerate}[(i)]
\item If $x \in G$ is an arbitrary element, then at least one of the triangles $(h_0^{-1}, x, e)$, $(g_0, h_0x, e)$ or $(h_0^{-1}g_0, x, e)$ is $\epsilon$-twisted.
\item If $x_n$ is any sequence of $G$-valued random variables, then for any $n$ one of the the following holds:
\begin{enumerate}[(a)]
 \item $\matP(|c(h_0^{-1}, x_n, e)|\geq\epsilon)\geq 1/3$, or
 \item $\matP(|c(g_0, h_0x_n, e)|\geq\epsilon)\geq 1/3$, or
 \item $\matP(|c(h_0^{-1}g_0, x_n, e)| \geq\epsilon)\geq 1/3$.
\end{enumerate}
\end{enumerate}
\end{cor}
\begin{proof} Clearly (ii) follows from (i). As for (i), it follows from Lemma \ref{tripletwist} applied to the triangle $(g_0, h_0, e)$ and the element $h_0x$ that at least one of the triangles $(h_0, e, h_0x)$, $(g_0, e, h_0x)$ or $(g_0, h_0, h_0x)$ is $\epsilon$-twisted.  Now $c$ is $G$-invariant and moreover alternating by Corollary \ref{CorAlternating}, and thus $|c(h_0, e, h_0x)|=|c(h_0^{-1}, x, e)|$, $|c(g_0, e, h_0x)|=|c(g_0, h_0x, e)|$, and $|c(g_0, h_0, h_0x)|=|c(h_0^{-1}g_0, x, e)|$, proving (i).
\end{proof}
The lemma motivates the further study of twisted triangles based at the identity. It turns out that another rather direct application of the cocycle identity provides a bound on the amount by which $c$ changes when one passes from such a triangle to an adjacent one.
\begin{lemma}\label{tetrahedron}
 Let $g, h, y \in G$. Then 
\[|c(g, y, e) - c(g, h, e)|\leq |c(h, y, e)|+|c(g^{-1}h, g^{-1}y, e)|.\]
%  at least one of the following holds.
%  
%  \begin{enumerate}
%   \item $|c(h, y, e)|\geq \epsilon$,
%   \item $|c(g^{-1}h, g^{-1}y, e)|\geq \epsilon$, or
%   \item $|c(g, y, e) - c(g, h, e)|\leq 2\epsilon$.
%  \end{enumerate}
\end{lemma}
\begin{proof} Applying the cocycle property of $c$ to the quadruple $h, g, y, e$ we get
\begin{equation}\label{Tetrahedron}
c(g, y, e) - c(h, y, e) + c(h, g, e) - c(h, g, y)=0.
\end{equation}
Using that $c$ is alternating (by Corollary \ref{CorAlternating}) and $G$-invariant we can rewrite this as
\begin{eqnarray*}
c(g, y, e) - c(g, h, e)&=& c(g, y, e) +   c(h, g, e)\\ &\overset{\eqref{Tetrahedron}}{=}& c(h, y,e) + c(h, g, y)\\ &=& c(h, y, e)-c(g^{-1}h, g^{-1}y, e).
\end{eqnarray*}
Now the lemma follows from the triangle inequality.
\end{proof}
\subsection{Consequences of laziness} The fact that we are working with a lazy random walk can be exploited to provide the following reduction of Theorem \ref{posprobtwist}.\begin{lemma}\label{h_i}
Assume that there exist elements $h_1, \dots, h_l \in G$ with the following property: For every $n, m \geq 1$ there exist $i, j \in \{1, \dots, l\}$ and $\epsilon > 0$ such that
\[\matP(|c(h_ix_n,h_jy_m,e)|\geq \epsilon)\geq \epsilon,\]
Then Theorem \ref{posprobtwist} holds.
\end{lemma}
\begin{proof}
Since the support of $\mu$ generates $G$ as a semigroup, for every $h \in G$ there exists $k_0 \geq 1$ with $h \in {\rm supp}(\mu)^{k_0} = {\rm supp}(\mu^{\ast k_0})$ and thus
$\mathbb P(x_{k_0} = h) =  \mu^{\ast k_0}(h) > 0$. Since $\mu(e) > 0$, this implies that $\matP(x_{k} = h)>0$ for all $k \geq k_0$. In particular, we can find $K \geq 1$ and $\delta > 0$ such that
\[
\matP(x_K=h_i)>\delta \quad \text{for all } i=1, \dots, l.
\]
Then, for each $i,j$ and $n,m\geq K+1$,
\[\matP(|c(x_n,y_m,e)|\geq \epsilon)\geq \delta \cdot \matP(|c(h_ix_{n-K},h_jy_{m-K},e)|\geq \epsilon),\]
 and the conclusion easily follows.
\end{proof}
\subsection{Proof of Theorem \ref{posprobtwist}}
Now we can finish the proof.
\begin{proof}[Proof of Theorem \ref{posprobtwist}]
%Recall that we fix the alternating homogeneous cocycle $c$.
%We will assume throughout that the theorem does not hold and arrive at a contradiction. We fix some symmetric generating set $S$ of $G$ and denote by $(x_n)$ the associated simple random walk on $G$, and let us denote by $(y_m)$ another independent simple random walk.
Set $N=\lceil 2\|c\|_\infty/\epsilon \rceil$ and fix $n \in \mathbb N$. Then one of the possibilities (a)-(c) in Corollary \ref{RWAlternative} holds. All three cases are similar, and we will assume that (a) holds. Moreover, we will assume that
\begin{equation}\label{FirstCase}
\matP(c(h_0^{-1}, x_n, e)\geq\epsilon)\geq 1/6.
\end{equation}
Indeed, if \eqref{FirstCase} fails then $\matP(c(h_0^{-1}, x_n, e)\leq-\epsilon)\geq 1/6$, and this case can be dealt with by a symmetric argument.

\noindent Now fix $m$ as well and define random events $A$ and $B$ by
\[
A := (\forall 0\leq i\leq N: \;c(h_0^{-1}, h_0^iy_m, e)\geq \epsilon/2),
\]
respectively
\[
B := (c(h_0^{-1}, x_n, e)\geq\epsilon \text{\ and\ } \forall 0\leq i\leq N:\;  |c(x_n,h_0^iy_m,e)|+|c(h_0x_n,h_0^{i+1}y_m,e)|\leq \epsilon/2).
\]
Observe that applying Lemma \ref{tetrahedron} with $g=h_0^{-1}, y=h_0^iy_m$, $ h=x_n$ we obtain
\[
|c(h_0^{-1}, h_0^iy_m, e) - c(h_0^{-1}, x_n, e)|\leq |c(x_n, h_0^iy_m, e)|+|c(h_0x_n, h_0^{i+1}y_m, e)|.
\]
Thus the event $B$ implies the event $A$, and hence $\matP(A) \geq \matP(B)$. We now assume for contradiction that  that for all $0\leq i\leq N$ we have
\[
\max\left\{\matP\left(|c(x_n,h_0^iy_m,e)|\geq \epsilon/4\right),\;\matP\left(|c(h_0x_n,h_0^{i+1}y_m,e)|\geq \epsilon/4\right)\right\} \leq \frac{1}{20(N+1)}, 
\]
and hence 
\begin{equation}\label{Estimate1}
\matP\left(|c(x_n,h_0^iy_m,e)|+|c(h_0x_n,h_0^{i+1}y_m,e)|\geq \epsilon/2\right)\geq \frac{1}{10(N+1)}.
\end{equation}
Using \eqref{FirstCase} and \eqref{Estimate1} we then obtain
\[
 \matP\left(\forall 0\leq i\leq N: \;c(h_0^{-1}, h_0^iy_m, e)\geq \epsilon/2\right) = \matP(A) \geq \matP(B) \geq 1/6-(N+1)\frac{1}{10{(N+1)}} > 0.
\]
We could thus find $y$ such that $c(h_0^{-1}, h_0^iy, e)\geq \epsilon/2$ for all $0\leq i \leq N$, but then applying Lemma \ref{HomogeneousSum} we end up with the contradiction
\[
\left|c(h_0^{-N-1}, y, e)\right| = \left|\sum_{i=0}^N c(h_0^{-1}, h_0^iy, e)\right| \geq (N+1)\epsilon/2 > \|c\|_\infty.
\]
This contradiction shows that necessarily there exists $0\leq i\leq N$ such that either
\[\matP(|c(x_n,h_0^iy_m,e)|\geq \epsilon/4)\geq \frac{1}{20(N+1)}\]
or
\[\matP(|c(h_0x_n,h_0^{i+1}y_m,e)|\geq \epsilon/4)\geq \frac{1}{20(N+1)}.\]
In either case we can conclude by Lemma \ref{h_i}.
\end{proof}

\section{Lazy random walks in acylindrically hyperbolic groups}\label{SecLRMAHG}

\subsection{Acylindrically hyperbolic groups and hyperbolically embedded subgroups}\label{SecAHG}
%Let $G$ be a group, and $Y\subset G$ a generating subset.
%\begin{definition}
%\end{definition}
Let $G$ be a group acting on a hyperbolic space $X$ by
isometries.  We say the action of $G$ on $X$ is \emph{non-elementary}
if $G$ contains two hyperbolic elements with disjoint pairs of fixed
points at infinity. We say that $G$ acts \emph{acylindrically} on $X$, if there are real valued functions $R$
and $N$ such that for every number $K \ge 0$, and for any pair of
points $x$ and $y$ in $X$ with $d_X(x, y) \ge R(K)$, there are at most
$N(K)$ group elements $g$ in $G$ such that $d_X(x, gx) \le K$ and
$d_X(y, gy) \le K$.

A group $G$ is \emph{acylindrically hyperbolic} if it admits a (possibly infinite) generating set $Y$ such that the associated Cayley graph ${\rm Cay}(G,Y)$ is hyperbolic and the action of $G$ on ${\rm Cay}(G,Y)$ is non-elementary and acylindrical \cite{Osin}. Such a generating set $Y$ is then called an \emph{acylindrical generating set}.

We now define the notion of hyperbolically embedded subgroup with respect to an acylindrical generating set. It is not the most general notion and not the original definition from \cite{DGO}, but we believe that it is more intuitive.

\begin{defn}(\cite[Theorem 3.9]{AMS}, \cite[Theorem 4.13]{Hull})
Let $G$ be a group and fix a (possibly infinite) generating set $Y$ of $G$ such that ${\rm Cay}(G,Y)$ is hyperbolic. Let $H < G$ be a finitely-generated subgroup and choose a finite generating set $S$ of $H$. Denote by $d_Y$ the word metric on $G$ with respect to $Y$ and by $d_S$ the word metric on $H$ with respect to $S$.
\begin{enumerate}
\item We say that $H$ is \emph{quasi-isometrically embedded} in $G$ with respect to $Y$ if there exist $\mu\geq 1$ and $c\geq 0$ such that for all $h\in H$ one has $d_S(1,h) \leq \mu d_Y(1,h)+c$.
\item  We say that $H$ is \emph{geometrically separated} in $G$ with respect to $Y$ if for every $\e>0$ there exists $R=R(\e)$ such that for $g\in G \setminus H$ we have
 \[{\rm diam}(H \cap \nei{gH}{\e}) < R,\]
where $ \nei{gH}{\e}$ denotes the $\epsilon$-neighborhood of $gH$ with respect to $d_Y$ and the diameter is measured with respect to $d_Y$.
\item We say that $H$ is \emph{hyperbolically embedded} in $G$ with respect to $Y$, denoted $H\h (G,Y)$,  if it is both quasi-isometrically embedded and geometrically separated in $G$ with respect to $Y$.
\end{enumerate}
We say that $H$ is \emph{hyperbolically embedded} in $G$ if it is hyperbolically embedded with respect to some generating set $Y$.
\end{defn}

It is proven in \cite[Theorem 6.14]{DGO} that any acylindrically hyperbolic group $G$ contains a unique maximal finite normal subgroup $K(G)$.

\subsection{Proof of the main theorem}
In order to conclude the proof of Theorem \ref{ThmMain} we are going to use the following result from \cite{MS}. We denote the free group on two generators by $F_2$.
\begin{thm}[Maher--Sisto]
\label{generichypemb}
 Let $G$ be an acylindrically hyperbolic group with acylindrical generating set $Y$ and maximal finite normal subgroup $K(G)$, and let $(x_n),(y_m)$ be independent lazy simple random walks on $G$. Let $A_{m,n}$ be the event that $\langle x_n,y_m,K(G)\rangle$ is
 %isomorphic to $F_2\ltimes K(G)$ and hyperbolically embedded in $(G,Y)$.
  \begin{itemize}
   \item isomorphic to $F_2\ltimes K(G)$ and
   \item hyperbolically embedded in $(G,Y)$.
  \end{itemize}
 Then $$\lim_{m,n \to \infty}\,\matP(A_{m,n})=1.$$
\end{thm}

We are now ready to prove a refinement of Theorem \ref{ThmMain}. Given an acylindrically hyperbolic group with acylindrical generating set $Y$ we denote by ${\rm HE}_2(G, Y)$ the collection of all virtually free-on-two-generators subgroups of $G$ that are hyperbolically embedded in $G$ with respect to $Y$, and by ${\rm HE}_2(G)$ the collection of all virtually free-on-two-generators subgroups of $G$ that are hyperbolically embedded in $G$. Then, by definition, we have $ {\rm HE}_2(G, Y) \subset {\rm HE}_2(G)$ for every acylindrical generating set $Y$, and thus the following generalizes Theorem \ref{ThmMain}.
\begin{thm}\label{MainRefinement}
 Let $G$ be an acylindrically hyperbolic group with acylindrical generating set $Y$. Then for every collection $\mathcal S$ of subgroups of $G$ which contains ${\rm HE}_2(G, Y)$ the natural restriction maps
\[
{\rm res}_{\mathcal S}: H^2_b(G; \R) \to \underset{H \in \mathcal S}{\lim_{\leftarrow}} H^2_b(H; \R) \quad \text{and} \quad {\rm res}_{\mathcal S}: EH^2_b(G; \R) \to \underset{H \in \mathcal S}{\lim_{\leftarrow}} EH^2_b(H; \R) 
\]
are both injective. \label{restr_gen}
\end{thm}

\begin{proof} Note first that the statement for $EH^2_b(G;\R)$ follows from the statement for $H^2_b(G;\R)$. Given $\alpha\in H^2_b(G) \setminus\{0\}$ we have to find $H\in{\rm HE}_2(G, Y)$ so that ${\rm res}_H(\alpha)\neq 0$. 

By Lemma \ref{LemmaBouarich}, $\alpha$ has a restriction-homogeneous representative $c$. (If $\alpha \in EH^2_b(G;\R)$ this just means that $c = c_f$ for a homogeneous quasimorphism $f$.)

The combination of Theorem \ref{generichypemb} and Theorem \ref{posprobtwist} gives that there exist $x,y$ so that
\begin{itemize}
 \item $E=\langle x,y\rangle$ is free on two generators,
 \item $H=E\ltimes K(G)\hookrightarrow_h (G,Y)$,
 \item $c(x,y,e)\neq 0$.
\end{itemize}
Corollary \ref{HomogeneousNonvanishing2} implies that ${\rm res}_E(\alpha)\neq 0$ and hence ${\rm res}_{H}(\alpha)\neq 0$, as required.\end{proof}

\section{A counterexample in degree $3$}\label{SecDegree3}
In this section we prove Proposition \ref{PropDeg3}, which says that our main result does not extend to degree $3$.

It follows from a classical result of Bloch that the third continuous-bounded cohomology $H^3_{cb}({\rm PSL}_2(\mathbb C))$ of ${\rm PSL}_2(\mathbb C)$ is generated by the bounded volume class (see e.g. \cite[Thm. 2]{BBI}). Explicitly, if $\omega$ denotes a volume form on $\mathbb{H}^3$ and $o\in \mathbb H^3$  denotes an arbitrary basepoint, then we obtain a bounded cocycle $c_{\omega, o}$ on ${\rm PSL}_2(\mathbb C)$ by 
\[
c_{\omega, o}(g_0, g_1, g_2, g_3) := \int_{\Delta(g_0.o, g_1.o, g_2.o, g_3.o)} \omega,
\]
where $\Delta(g_0.o, g_1.o, g_2.o, g_3.o) \subset \mathbb H^3$ denotes the geodesic tetrahedron spanned by the points $g_0.o$, $g_1.o$, $g_2.o$ and $g_3.o$. The class ${\rm Vol}_b := [c_{\omega, o}]$ then does not depend on the choice of basepoint, and by the aforementioned result of Bloch generates $H^3_{cb}({\rm PSL}_2(\mathbb C))$. Given a discrete subgroup $G<{\rm PSL}_2(\mathbb C)$ we can consider the restriction $\alpha_G := {\rm Vol}_b|_G \in H^3_b(G)$. Soma \cite{Soma} showed that $\alpha_G\neq 0$ if and only if $G$ is geometrically infinite.

If $G$ is the fiber of the mapping torus of a pseudo-Anosov of a closed surface, then $G$ is geometrically infinite but any infinite index subgroup, in particular any free subgroup, is geometrically finite \cite{ScottSwarup}. In particular, $\alpha_G$ is non-trivial, but its restriction to any free subgroup of $G$ is trivial.

\end{document}